\documentclass{amsart}
\usepackage[all]{xy}
\usepackage{amssymb}
\usepackage{color}

\newcommand{\doublegroupoid}[4]{\xymatrix{#1  \ar@<2pt>[d] \ar@<-2pt>[d] & #2 \ar@<2pt>[l] \ar@<-2pt>[l] \ar@<2pt>[d] \ar@<-2pt>[d] \\ #3  & #4 \ar@<2pt>[l] \ar@<-2pt>[l]}}
\newcommand{\barW}{\overline{W}}
\newcommand{\bitimes}[2]{\,_{#1}\!\!\times_{#2}}
\newcommand{\squarelabel}[5]{\xymatrix@M=0pt{ \ar@{}[dr]|{#1}  & \ar_{#2}[l] \\ \ar^{#5}[u] &  \ar^{#4}[l] \ar_{#3}[u] }}
\newcommand{\squarenolabel}[1]{\xymatrix@M=0pt{ \ar@{}[dr]|{#1} \ar[d] & \ar[l] \ar[d] \\  &  \ar[l]}}
\newcommand{\arrows}{\rightrightarrows}
\newcommand{\threenerve}[3]{\xymatrix{#3 \ar[d] \ar@<4pt>[d] \ar@<-4pt>[d] \\ #2 \ar@<2pt>[d] \ar@<-2pt>[d] \\ #1}}
\newcommand{\fournerve}[4]{\xymatrix{#4 \ar@<2pt>[d] \ar@<-2pt>[d] \ar@<6pt>[d] \ar@<-6pt>[d]\\#3 \ar[d] \ar@<4pt>[d] \ar@<-4pt>[d] \\ #2 \ar@<2pt>[d] \ar@<-2pt>[d] \\ #1}}

\newcommand{\mapstoxy}{\; \raisebox{-1.75pc}{$\longmapsto$} \;}

\newcommand{\largesim}{\mbox{\large $\thicksim$}}

\newcommand{\defequal}{:=}
\newcommand{\tolabel}[1]{\xrightarrow{#1}}
\newcommand{\id}{\mathrm{id}}
\newcommand{\intersect}{\cap}
\DeclareMathOperator{\diag}{diag}
\DeclareMathOperator{\Tot}{Tot}

\definecolor{forest}{rgb}{0.1,0.3,0}

\newcommand{\comment}[1]{}

\newtheorem{thm}{Theorem}[section]
\newtheorem{prop}[thm]{Proposition}
\newtheorem{lemma}[thm]{Lemma}

\theoremstyle{definition}
\newtheorem{definition}[thm]{Definition}

\theoremstyle{remark}
\newtheorem{remark}[thm]{Remark}
\newtheorem{remarks}[thm]{Remarks}

\newtheorem{example}[thm]{Example}

\newtheorem*{ack}{Acknowledgements}
\numberwithin{equation}{section}

\newcommand{\suchthat}{:}

\begin{document}

\title{From double Lie groupoids to local Lie $2$-groupoids}
\author{Rajan Amit Mehta}
\address{
Department of Mathematics \\ Pennsylvania State University \\ State
College, PA 16802} \email{mehta@math.psu.edu}

\author{Xiang Tang}
\address{Department of Mathematics\\
Washington University in Saint Louis\\
One Brookings Drive\\
Saint Louis, Missouri, USA 63130}
\email{xtang@math.wustl.edu}

\begin{abstract}
We apply the bar construction to the nerve of a double Lie groupoid to obtain a local Lie $2$-groupoid. As an application, we recover Haefliger's fundamental groupoid from the fundamental double groupoid of a Lie groupoid. In the case of a symplectic double groupoid, we study the induced closed $2$-form on the associated local Lie $2$-groupoid, which leads us to propose a definition of a symplectic $2$-groupoid.
\end{abstract}

\maketitle
\section{Introduction}
In homological algebra, given a bisimplicial object $A_{\bullet,\bullet}$ in
an abelian category, one naturally associates two
chain complexes. One is the diagonal complex
$\diag(A)\defequal \{A_{p,p}\}$, and the other is the total complex
$\Tot(A)\defequal \{\sum_{p+q=\bullet} A_{p,q}\}$. The
(generalized) Eilenberg-Zilber theorem \cite{dol-pup} (see, e.g.\ \cite[Theorem 8.5.1]{weibel}) states that $\diag(A)$ is quasi-isomorphic to $\Tot(A)$. This result
plays a fundamental role in the proof of the K\"unneth theorem.

The Eilenberg-Zilber theorem has recently been further generalized by
Cegarra and Remedios \cite{ce-me:bar} as follows. Given a bisimplicial set
$X_{\bullet,\bullet}$, one naturally associates two simplicial sets. One is
the diagonal $\diag(X):=\{X_{p,p}\}$, and the other is given by Artin and Mazur's \emph{bar construction} $\barW$. Cegarra and Remedios proved that these two simplicial sets are
weakly homotopy equivalent. More recently, Cegarra, Heredia, and Remedios \cite{ce-he-me:double} studied the diagonal construction on (discrete) double groupoids satisfying a natural ``filling condition''. In this paper, we study the behavior of the bar construction in the smooth category, particularly with respect to bisimplicial manifolds arising from double Lie groupoids.

A double Lie groupoid \cite{bro-mac, mac:dblie1} is essentially a groupoid object in the category of Lie groupoids. We can present a double
Lie groupoid as a square
\begin{equation}\label{eqn:dbliegpd}
\doublegroupoid{V}{D}{M}{H},
\end{equation}
where each edge is a Lie groupoid, and where the various groupoid structures satisfy certain compatibility conditions (The precise definition will be recalled in \S\ref{sec:dbliedef}.).  Just as an element of a Lie groupoid can be depicted as an arrow, an element of (the total space of) a double Lie groupoid can be depicted as a square
\begin{equation}\label{eqn:square}
 \squarelabel{\alpha}{t_V \alpha}{s_H \alpha}{s_V \alpha}{t_H \alpha}
\end{equation}
whose edges correspond to
 the images of $\alpha \in D$ under the source and target
maps. Just as a Lie groupoid $G$ induces a simplicial manifold $N_\bullet G$ (called the \emph{nerve} of $G$), where $N_p G$ consists of composable $p$-tuples of arrows in $G$, a double Lie groupoid \eqref{eqn:dbliegpd} induces a bisimplicial manifold $N_{\bullet,\bullet} D$, where $N_{p,q}D$ consists of $p \times q$ arrays of composable squares.

We apply the bar construction to the nerve $N_{\bullet,\bullet}D$ of a
double Lie groupoid and study the Kan properties of the resulting simplicial
manifold $\barW ND$. We prove that $\barW ND$ satisfies the definition
\cite{zhu:kan} of a (weak) \emph{local Lie
$2$-groupoid}. Furthermore, if $D$ satisfies a filling condition then $\barW ND$ is a genuine Lie $2$-groupoid.

We then describe a functor $\mathfrak{G}$ from the category of (local) Lie
$2$-groupoids to the category of topological groupoids, essentially given by Kan replacement (in the local case) \cite{zhu:kan} followed by truncation. We thus have the following diagram of functors:
\begin{equation*}
     \fbox{double Lie groupoids} \tolabel{\barW N} \fbox{(local) Lie $2$-groupoids} \tolabel{\mathfrak{G}} \fbox{topological groupoids}
\end{equation*}

As an application, we look at the following example of a \emph{fundamental
double Lie groupoid}. If $G\rightrightarrows M$ is a Lie groupoid where the source map is a fibration, then the fundamental groupoids of $G$ and
$M$ form a double Lie groupoid
\begin{equation*}
\doublegroupoid{G}{\Pi_1(G)}{M}{\Pi_1(M)}.
\end{equation*}
The fibration assumption implies that the filling condition is satisfied, so $\barW N\Pi_1(G)$ is a Lie $2$-groupoid. By truncating this Lie $2$-groupoid, we recover the Haefliger fundamental groupoid \cite{haefliger:homotopy, haefliger:orbi, mo-mr:fundamental} of $G$.

When the source map of $G$ is not a
fibration, $\Pi_1(G)$ does not form a double Lie groupoid (instead, it is in some sense a ``groupoid in the category of local Lie groupoids.''). However, we can still construct a bisimplicial manifold $N_{\bullet,\bullet}\Pi_1(G)$, to which we can still apply the bar construction and obtain a local Lie $2$-groupoid. When we apply the functor $\mathfrak{G}$ to this local Lie $2$-groupoid, we again recover the Haefliger fundamental groupoid of $G$. Thus we are able to connect two natural notions of ``fundamental groupoid of a Lie groupoid''.

As an application to Poisson geometry, we also consider the case of \emph{symplectic double groupoids} \cite{mac:symplectic}, i.e.\ double Lie groupoids \eqref{eqn:dbliegpd} where $D$ is equipped with a symplectic form making $D$ a symplectic groupoid over both $V$ and $H$. Symplectic double groupoids are the objects that integrate Poisson groupoids and Lie bialgebroids \cite{ste:lagpd, weinstein:coisotropic}. If $D$ is a symplectic double groupoid integrating a Lie bialgebroid $(A, A^*)$, then we may think of $\barW ND$ as the (local) Lie $2$-groupoid integrating the Courant algebroid $A \oplus A^*$. This method can be used to construct many interesting examples of Lie $2$-groupoids integrating Courant algebroids.

Of course, in the case of a symplectic double groupoid $D$, it is natural to ask what geometric structure on $\barW ND$ is induced from the symplectic form $\omega$ on $D$. We find that $\omega$ naturally induces a $2$-form $\Omega$ on $\barW_2 ND$, which is both closed and multiplicative; however, $\Omega$ is usually degenerate.

At first, we had hoped to extend $\Omega$ to a non-degenerate $2$-form.
However, it seems that such an extension does not exist
naturally, as the following example illustrates. Let $M$ be a manifold, and consider the symplectic double groupoid
\[
\doublegroupoid{M\times M}{T^* M\times T^*M}{M}{T^*M}.
\]
Up to level $2$, the associated Lie $2$-groupoid is of the form
\begin{equation}\label{eqn:standard2}
\threenerve{M}{M\times T^*M}{M\times T^*M \times (T^*M\oplus T^*M)}.
\end{equation}
This is the Lie $2$-groupoid that integrates the standard Courant algebroid $TM \oplus T^*M$.
There does not seem to be a canonical\footnote{However, Li-Bland has explained to us that he and \v{S}evera \cite{lbs} have a model for locally integrating $TM\oplus T^*M$ where there is a nondegenerate (but also noncanonical) closed $2$-form on the $2$nd level.} non-degenerate $2$-form on the $2$nd level $M\times T^*M \times (T^*M\oplus T^*M)$, so instead we seek conditions that control the degeneracy of $\Omega$.

Our approach is inspired by Xu's definition of quasi-symplectic groupoid \cite{xu:quasi-symp}. The idea is to localize to various degenerate submanifolds of $\barW_2 ND$ and then consider the compatibility properties between the face maps and the $2$-form $\Omega$ along those submanifolds. Based on these compatibility properties, we tentatively propose a definition of \emph{symplectic $2$-groupoid}.  We prove that the nondegeneracy of the multiplicative $2$-form $\omega$ on $D$ is equivalent to the nondegeneracy conditions for $\Omega$ in our definition of (local) symplectic $2$-groupoid.

We believe that the conditions that make up our definition of symplectic $2$-groupoid extract important information about a symplectic double groupoid (and its associated infinitesimal object) from its associated (local) Lie $2$-groupoid.  We plan to discuss this issue in a future work.

We should emphasize that we do not consider our definition of symplectic $2$-groupoid to be the final word, but rather a way of beginning a discussion on the problem. We have learned from Li-Bland and \v{S}evera \cite{lbs} that they have discovered various possible solutions that are different from ours. We hope that together our ideas will lead to a full understanding of symplectic $2$-groupoids.

This paper is organized as follows. In \S\ref{sec:bisimplicial}, we briefly review the definition of bisimplicial manifolds and the bar construction. In \S\ref{sec:dblie}, we recall the definition of double Lie groupoid, and we describe the constructions of the nerve and classifying space of a double Lie groupoid. In \S\ref{sec:n}, we prove that $\barW$ of the nerve of a double Lie groupoid is a local Lie $2$-groupoid. The functor $\mathfrak{G}$ and the example of the fundamental double groupoid are discussed in \S\ref{sec:groupoidization}. We turn to symplectic double groupoids in \S\ref{sec:symplectic} and conclude with our proposal for the definition of a symplectic $2$-groupoid.

\begin{ack}
We would like to thank A.\ Cattaneo, D.\ Li-Bland, P.\ Severa, and C.\
Zhu for interesting discussions and suggestions. Tang's research is
partially supported by NSF grant 0900985.
\end{ack}

\section{Bisimplicial manifolds}\label{sec:bisimplicial}

In this section we briefly recall the definitions of simplicial and bisimplicial manifolds (see, e.g.\ \cite{gel-man, goe-jar}), as well as the bar construction of Artin and Mazur \cite{art-maz}.

\subsection{Definitions}

\begin{definition} \label{dfn:simplicial}A \emph{simplicial manifold} is a sequence $X_\bullet = \{X_q\}$, $q \geq 0$, of manifolds equipped with surjective submersions $f_i^q: X_q \to X_{q-1}$ (called \emph{face maps}), $i=0,\dots,q$,  and embeddings $\delta_i^q : X_q \to X_{q+1}$ (called \emph{degeneracy maps}), $i=0, \dots, q$,  such that
\begin{align}
f_i^{q-1}f_j^q &= f_{j-1}^{q-1}f_i^q, &i < j,\label{eqn:twoface}\\
\delta_i^{q+1}\delta_j^q &= \delta_{j+1}^{q+1}\delta_i^q, &i < j, \label{eqn:twodegen}\\
f_i^{q+1}\delta_j^q &= \left\{ \begin{aligned}
&\delta_{j-1}^{q-1}f_i^q, & i &< j,\\
&id, & i &= j, \; i = j+1, \\
&\delta_j^{q-1}f_{i-1}^q, & i &> j+1. \label{eqn:facedeg}\end{aligned}
\right.
\end{align}
\end{definition}

\begin{definition} \label{dfn:bisimplicial} A \emph{bisimplicial manifold} is a two-parameter sequence $X_{\bullet,\bullet} = \{X_{p,q}\}$, $p, q \geq 0$ of manifolds equipped with
\begin{itemize}
 \item surjective submersions $v_i^{p,q}: X_{p,q} \to X_{p, q-1}$ (called \emph{vertical face maps}), $i=0,\dots,q$,
\item surjective submersions $h_i^{p,q}: X_{p,q} \to X_{p-1, q}$ (called \emph{horizontal face maps}), $i=0,\dots,p$,
\item embeddings $\mu_i^{p,q} : X_{p,q} \to X_{p,q+1}$ (called \emph{vertical degeneracy maps}), $i=0, \dots, q$, and
\item embeddings $\eta_i^{p,q} : X_{p,q} \to X_{p+1,q}$ (called \emph{horizontal degeneracy maps}), $i=0, \dots, p$,
\end{itemize}
such that
\begin{enumerate}
 \item $\left( X_{p,\bullet}, \{v_i^{p,q}\}, \{\mu_i^{p,q}\} \right)$ is a simplicial manifold for each $p$,
\item $\left( X_{\bullet,q}, \{h_i^{p,q}\}, \{\eta_i^{p,q}\} \right)$ is a simplicial manifold for each $q$, and
\item the horizontal and vertical structure maps commute:
    \begin{enumerate}
    \item $v_j^{p-1,q} h_i^{p,q} = h_i^{p,q-1} v_j^{p,q}$,
    \item $\mu_j^{p+1,q} \eta_i^{p,q} = \eta_i^{p,q+1} \mu_j^{p,q}$,
    \item $v_j^{p+1,q} \eta_i^{p,q} = \eta_i^{p,q-1} v_j^{p,q}$,
    \item $\mu_j^{p-1,q} h_i^{p,q} = h_i^{p,q+1} \mu_j^{p,q}$.
    \end{enumerate}

\end{enumerate}
\end{definition}

\subsection{The bar construction}\label{sec:bar}

The \emph{bar construction} is a functor $\barW$ from the category of bisimplicial manifolds to the category of simplicial manifolds. It is described as follows.

Let $X_{\bullet,\bullet}$ be a bisimplicial manifold. For each $r$, let $\barW_r X$ be the space consisting of $r$-tuples $(x_0, \dots, x_r)$, where $x_i \in X_{i, r-i}$, and where $v_0^{i,r-i}(x_i) = h_{i+1}^{i+1,r-i-1}(x_{i+1})$ for $0 \leq i < r$. In other words, $\barW_r X$ is the fiber product
\begin{equation}\label{eqn:bar}
    \barW_r X = X_{0,r} \bitimes{v_0}{h_1} X_{1,r-1} \bitimes{v_0}{h_2} \cdots \bitimes{v_0}{h_r} X_{r,0}.
\end{equation}
Since both the horizontal and vertical face maps are submersions, we have that $\barW_r X$ is a smooth manifold for each $r$.

The sequence $\barW_\bullet X = \{ \barW_r X \}$ is a simplicial manifold with face and degeneracy maps defined as follows:
\begin{align}
 f_i^r (x_0,\dots,x_r) &= \left( v_i (x_0), v_{i-1}(x_1), \dots, v_1(x_{i-1}), h_i(x_{i+1}), \dots, h_i(x_r) \right), \label{eqn:barface} \\
\delta_i^r (x_0,\dots,x_r) &= \left( \mu_i(x_0), \mu_{i-1}(x_1), \dots, \mu_0(x_i), \eta_i(x_i), \dots, \eta_i(x_r) \right). \label{eqn:bardegen}
\end{align}

\section{Double Lie groupoids}\label{sec:dblie}

In this section, we review the definition of double Lie groupoids \cite{bro-mac, mac:dblie1} and describe the nerve functor from the category of double Lie groupoids to that of bisimplicial manifolds. Finally, we discuss various constructions for the classifying space of a double Lie groupoid.

\subsection{Definition and examples}\label{sec:dbliedef}

Consider a square \eqref{eqn:dbliegpd} where each edge is a Lie groupoid. Let $s_H$ and $t_H$ denote the source and target maps from $D$ to $H$, respectively, and let $s_V$ and $t_V$ denote those from $D$ to $V$. We will use $s$ and $t$ to denote the source and target maps from either $H$ or $V$ to $M$; the domain will be clear from the context.

Let $\cdot_H : D \bitimes{s_H}{t_H} D \to D$ and $\cdot_V : D \bitimes{s_V}{t_V} D \to D$ denote the multiplication maps for the Lie groupoid structures on $D$ over $H$ and $V$, respectively.

\begin{definition}\label{dfn:dbliegpd}
The square \eqref{eqn:dbliegpd} is a \emph{double Lie groupoid} if the following conditions hold:
\begin{enumerate}
\item The horizontal and vertical source and target maps commute:
\begin{align*}
 s \circ s_H = s \circ s_V, && t \circ t_H = t \circ t_V,\\
 t \circ s_H = s \circ t_V, && s \circ t_H = t \circ s_V.
\end{align*}
\item The multiplication maps respect source and target:
\begin{align*}
     s_V (\alpha_1 \cdot_H \alpha_2) = (s_V \alpha_1) \cdot (s_V \alpha_2), && t_V (\alpha_1 \cdot_H \alpha_2) = (t_V \alpha_1) \cdot (t_V \alpha_2), \\
    s_H (\alpha_1 \cdot_V \alpha_3) = (s_H \alpha_1) \cdot (s_H \alpha_3), && t_H (\alpha_1 \cdot_V \alpha_3) = (s_H \alpha_1) \cdot (s_H \alpha_2),
\end{align*}
for all $\alpha_i \in D$ such that $s_H \alpha_1 = t_H \alpha_2$ and $s_V \alpha_1 = t_V \alpha_3$.
\item The \emph{interchange law}
\begin{equation*}
     \left(\alpha_{11} \cdot_H \alpha_{12}\right) \cdot_V \left(\alpha_{21} \cdot_H \alpha_{22}\right) = \left(\alpha_{11} \cdot_V \alpha_{21}\right) \cdot_H \left(\alpha_{12} \cdot_V \alpha_{22}\right)
\end{equation*}
holds for all $\alpha_{11}, \alpha_{12}, \alpha_{21}, \alpha_{22}
\in D$ such that $s_H(\alpha_{i1}) = t_H(\alpha_{i2})$ and
$s_V(\alpha_{1i}) = t_V(\alpha_{2i})$ for $i=1,2$.
\item The double-source map $(s_V, s_H): D \to V \bitimes{s}{s} H$ is a submersion.
\end{enumerate}
\end{definition}

We note that, in \cite{bro-mac} and \cite{mac:dblie1}, the double-source map is required to be a \emph{surjective} submersion. As noted by Stefanini \cite{ste:lagpd}, the surjectivity property does not play an essential role in the definition of double Lie groupoids, so we follow his lead in dropping the requirement. However, we will see that double Lie groupoids with surjective double-source map behave particularly well with respect to the main construction of this paper. Thus we give this property a name:

\begin{definition}
     A double Lie groupoid is \emph{full} if the double-source map $(s_V, s_H): D \to V \bitimes{s}{s} H$ is surjective.
\end{definition}

\begin{remarks}
 \begin{enumerate}
  \item Following Mackenzie\footnote{We note that our conventions for the square \eqref{eqn:square} are different from those in \cite{mac:dblie1}. In our conventions, the left and right sides of the square \eqref{eqn:square} represent arrows in $H$, and the top and bottom sides represent arrows in $V$. We can explain this convention by observing that horizontal (resp.\ vertical) composition of squares takes place over elements of $H$ (resp.\ $V$).} \cite{mac:dblie1}, we may depict any element $\alpha \in D$ as a square \eqref{eqn:square}. Condition (1) in Definition \ref{dfn:dbliegpd} expresses the fact that the vertices of the square are well-defined points in $M$. The multiplication operations $\cdot_H$ and $\cdot_V$ may be depicted, respectively, as horizontal and vertical composition of squares. Condition (2) implies that four-fold products of the form
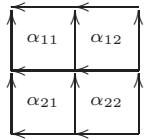
\begin{equation*}
\xymatrix@M=0pt{& \ar[l]& \ar[l]\\ \ar[u]& \ar[l] \ar[u] \ar@{}[ul]|{\alpha_{11}}& \ar[l] \ar[u] \ar@{}[ul]|{\alpha_{12}}\\ \ar[u]& \ar[l] \ar[u] \ar@{}[ul]|{\alpha_{21}}& \ar[l] \ar[u] \ar@{}[ul]|{\alpha_{22}}}
\end{equation*}
can be computed in two different ways (horizontally then vertically, and vice versa), and the interchange law guarantees that the two four-fold products agree.

\item Conditions (1) and (2) in Definition \ref{dfn:dbliegpd} imply that $D \bitimes{s_V}{t_V} D$ and $D \bitimes{s_H}{t_H} D$ naturally inherit groupoid structures over $H \bitimes{s}{t} H$ and $V \bitimes{s}{t} V$, respectively. Condition (4) guarantees that the source maps $(s_H, s_H): D \bitimes{s_V}{t_V} D \to H \bitimes{s}{t} H$ and $(s_V,s_V): D \bitimes{s_H}{t_H} D \to V \bitimes{s}{t} V$ are submersions, so that $D \bitimes{s_V}{t_V} D$ and $D \bitimes{s_H}{t_H} D$ are in fact \emph{Lie} groupoids.

\item It follows from the conditions of Definition \ref{dfn:dbliegpd} that the structure maps (source, target, multiplication, identity, and inverse) for either of the groupoid structures on $D$ are Lie groupoid morphisms with respect to the other groupoid structure. In this sense, a double Lie groupoid is a ``Lie groupoid object in the category of Lie groupoids.''
 \end{enumerate}
\end{remarks}

\begin{example}[Pair double groupoid]
     Let $G \arrows M$ be a Lie groupoid. Then
\begin{equation*}
     \doublegroupoid{G}{G \times G}{M}{M \times M},
\end{equation*}
where the horizontal edges are pair groupoids, is a full double Lie groupoid.
\end{example}

\begin{example}[Fundamental double Lie groupoid]\label{example:fundamental}
Let $M$ be a manifold. The \emph{fundamental groupoid} $\Pi_1(M) \arrows M$ consists of equivalence classes of paths in $M$ modulo endpoint-preserving homotopies. The operation $M \mapsto \Pi_1(M)$ is a functor from the category of manifolds to the category of Lie groupoids; given a smooth map of manifolds $\phi: M \to N$, the induced map of fundamental groupoids is denoted $\phi_*: \Pi_1(M) \to \Pi_1(N)$.

Let $G \arrows M$ be a Lie groupoid. The source and target maps $s,t: G \to M$ induce maps $s_*, t_*: \Pi_1(G) \to \Pi_1(M)$. If $s$ (or, equivalently, $t$) is a fibration, then the fiber product $\Pi_1(G) \bitimes{s_*}{t_*} \Pi_1(G)$ can be canonically identified with $\Pi_1(G \bitimes{s}{t} G)$.  Thus the map $\Pi_1(G \bitimes{s}{t} G) \to \Pi_1(G)$ induced by multiplication in $G$ can be interpreted as a multiplication map for $\Pi_1(G)$. In this manner, $\Pi_1(G)$ has the structure of a Lie groupoid over $\Pi_1(M)$, and it is fairly easy to see that
\begin{equation}\label{eqn:fundamental}
     \doublegroupoid{G}{\Pi_1(G)}{M}{\Pi_1(M)}
\end{equation}
is a double Lie groupoid. Again using the assumption that $s$ is a fibration, we have that this double Lie groupoid is full.

The fibration property that we have used is nontrivial and excludes common examples such as \v{C}ech groupoids and certain groupoids representing orbifolds. Later, in \S\ref{sec:fundamental}, we describe how Lie groupoids not satisfying the fibration property may still be incorporated into our construction.
\end{example}

\begin{example}[Symplectic double groupoids]\label{example:symplectic}
A \emph{symplectic double groupoid} is a double Lie groupoid \eqref{eqn:dbliegpd}, where $D$ is equipped with a symplectic form making $D$ into a symplectic groupoid with respect to both groupoid structures. For example, if $G \arrows M$ is a Lie groupoid with Lie algebroid $A \to M$, then the cotangent bundle $T^*G$ can be viewed as a symplectic double groupoid as follows:
\begin{equation*}
     \doublegroupoid{G}{T^*G}{M}{A^*}.
\end{equation*}
Here, the groupoid structures on the horizontal sides are simply vector bundles.

Another example is the pair double groupoid
\begin{equation*}
     \doublegroupoid{G}{G \times \bar{G}}{M}{M \times \bar{M}},
\end{equation*}
where $G \arrows M$ is a symplectic groupoid. The bars over the rightmost copies of $G$ and $M$ indicate that they are equipped with the opposite symplectic/Poisson structures.

More generally, any Poisson groupoid can, under certain assumptions, be integrated to a symplectic double groupoid \cite{ste:lagpd, weinstein:coisotropic}. Conversely, any symplectic double groupoid can be differentiated in two different ways to obtain Poisson groupoids in duality \cite{lu-weinstein, mac:symplectic}.
\end{example}

\begin{example}
     Double Lie groupoids of the form
\begin{equation*}
     \doublegroupoid{G}{D}{M}{M}
\end{equation*}
are in one-to-one correspondence with strict Lie $2$-groupoids. Such double Lie groupoids are always full.
\end{example}

\subsection{The nerve functor}\label{sec:nerve}

Consider a double Lie groupoid as in \eqref{eqn:dbliegpd}. For $q>0$, let $V^{(q)}$ be the space of composable $q$-tuples of elements of $V$:
\begin{equation*}
     V^{(q)} = \{ (\theta_1, \dots, \theta_q) \in V^q \suchthat s(\theta_i) = t(\theta_{i+1}) \}.
\end{equation*}
We have that $V^{(q)}$ is smooth since $s$ is a submersion. Next,
define $D^{(q)}_H$ to be the space of horizontally composable
$q$-tuples of elements of $D$:
\begin{equation*}
     D^{(q)}_H = \{ (\alpha_1, \dots, \alpha_q) \in D^q \suchthat s_H(\alpha_i) = t_H(\alpha_{i+1}) \}.
\end{equation*}
Smoothness of $D^{(q)}_H$ follows from the fact that $s_H$ is a submersion.
\begin{lemma}
     For every $q>0$, the space $D^{(q)}_H$ is a Lie subgroupoid of the Cartesian product $D^q \arrows V^q$, with base $V^{(q)}$.
\end{lemma}
\begin{proof}
     The fact that $D^{(q)}_H \arrows V^{(q)}$ is a subgroupoid of $D^q \arrows V^q$ follows from conditions (1) and (2) in Definition \ref{dfn:dbliegpd}. The fact that the source map $s_V^q : D^{(q)}_H \arrows V^{(q)}$ is a submersion follows from condition (4) in Definition \ref{dfn:dbliegpd}.
\end{proof}

Similarly, for $p>0$, we define $H^{(p)}$ and $D^{(p)}_V$ as follows:
\begin{align*}
         H^{(p)} &= \{ (\eta_1, \dots, \eta_p) \in H^p \suchthat s(\eta_i) = t(\eta_{i+1}) \}, \\
     D^{(p)}_V &= \{ (\alpha_1, \dots, \alpha_p) \in D^p \suchthat s_V(\alpha_i) = t_V(\alpha_{i+1}) \}.
\end{align*}
\begin{lemma}
     For every $p>0$, the space $D^{(p)}_V$ is a Lie subgroupoid of the Cartesian product $D^p \arrows H^p$, with base $H^{(p)}$.
\end{lemma}

Now we are ready to describe the nerve functor from the category of double Lie groupoids to the category of bisimplicial manifolds. Consider a double Lie groupoid as in \eqref{eqn:dbliegpd}. For $p,q > 0$, let $N_{p,q} D$ consist of all $pq$-tuples $\{\alpha_{ij}\}$, where $\alpha_{ij} \in D$ for $1\leq i \leq p$ and $1 \leq j \leq q$, subject to the following compatibility conditions:
\begin{align}\label{eqn:ncompat}
     s_H(\alpha_{ij}) &= t_H(\alpha_{i(j+1)}), & s_V(\alpha_{ij}) &= t_V(\alpha_{(i+1)j}).
\end{align}
We take $N_{0,q}D = V^{(q)}$ and $N_{p,0}D = H^{(p)}$ for $p,q > 0$, and we set $N_{0,0}D = M$.

An element of $N_{p,q}D$ (for $p,q > 0$) may be depicted as a composable $p \times q$ rectangular array of elements of $D$. For example, a typical element of $N_{3,2}D$ is of the form
\begin{equation*}
\xymatrix@M=0pt{& \ar[l]& \ar[l]\\ \ar[u]& \ar[l] \ar[u] \ar@{}[ul]|{\alpha_{11}}& \ar[l] \ar[u] \ar@{}[ul]|{\alpha_{12}}\\ \ar[u]& \ar[l] \ar[u] \ar@{}[ul]|{\alpha_{21}}& \ar[l] \ar[u] \ar@{}[ul]|{\alpha_{22}}\\ \ar[u]& \ar[l] \ar[u] \ar@{}[ul]|{\alpha_{31}}& \ar[l] \ar[u] \ar@{}[ul]|{\alpha_{32}}}.
\end{equation*}
Such an array may be viewed as either a composable $p$-tuple of elements of $D^{(q)}_H$ or as a composable $q$-tuple of elements of $D^{(p)}_V$. Thus there are natural horizontal and vertical simplicial manifold structures on $N_{\bullet,\bullet} D$, obtained by identifying $N_{\bullet, q} D$ (respectively, $N_{p,\bullet}D$) with the nerve of the Lie groupoid $D^{(q)}_H \arrows V^{(q)}$ (respectively, $D^{(p)}_V \arrows H^{(p)}$). The interchange law then implies that the horizontal and vertical structure maps commute. Thus, we have:
\begin{prop}
     $N_{\bullet,\bullet} D$ is a bisimplicial manifold.
\end{prop}

\subsection{Classifying spaces}
Let $\Delta^\bullet$ be the cosimplicial space of simplices, with the standard coface maps $\hat{f}_i^q : \Delta^{q-1} \to \Delta^q$ and codegeneracy maps $\hat{\delta}_i^q: \Delta^{q+1} \to \Delta^q$, satisfying equations dual to \eqref{eqn:twoface}--\eqref{eqn:facedeg}. Recall that
the geometric realization $|X|$ of a simplicial space $X_\bullet$ is
defined to be
\[
\coprod_{n=0}^\infty(\Delta^n\times X_n)/\largesim,
\]
where the equivalence relation $\sim$ is generated by
$(\hat{f}_i^q (s_{q-1}),x_q)\sim (s_{q-1},f_i^q(x_q))$ and
$(\hat{\delta}^q_i(s_{q+1}), x_q)=(s_{q+1}, \delta^q_i(x_q))$, for $s_{q \pm 1} \in
\Delta^{q \pm 1} , x_q \in X_q$. Geometric realization is a
functor from the category of simplicial spaces to the category of
topological spaces.

For a bisimplicial space $X_{\bullet, \bullet}$, there are three
natural ways to define the geometric realization:
\begin{enumerate}
\item Diagonal: Consider the diagonal simplicial space $X^D_\bullet$, where $X^D_n = X_{nn}$ and the structure maps are obtained by composing corresponding horizontal and vertical structure maps. Take the geometric realization $|X^D|$.
\item Horizontal-Vertical: For a fixed $p$, consider the ``vertical'' simplicial space $X_{p,\bullet}$. Let $X^H_\bullet$ be the simplicial space where $X^H_p = |X_{p,\bullet}|$, and where the structure maps are inherited from the horizontal simplicial structure on $X_{\bullet, \bullet}$. Take the geometric realization $|X^H|$.
\item Vertical-Horizontal: Similarly, let $X^V_\bullet$ be the simplicial space where $X^V_q = |X_{\bullet, q}|$. Take the geometric realization $|X^V|$.
\end{enumerate}
A version of the Eilenberg-Zilber theorem \cite{gel-man} states that the
topological spaces $|X^D|$, $|X^H|$, and $|X^V|$ are all
homeomorphic, so we are justified in simply writing $|X|$.

Let $D$ be a double Lie groupoid. As was explained in
\S\ref{sec:nerve}, the nerve functor defines a bisimplicial space
$ND$. (It is in fact a bisimplicial manifold, but the smooth
structure is not relevant for the present discussion.) We define the
\emph{classifying space} $BD$ to be the geometric realization
$|ND|$.

In the case of the nerve of a double Lie groupoid $D$, the latter two constructions of the geometric realization correspond to applying the usual (groupoid) classifying space functor to $D$ in both directions; in other words, $BD$ may be computed as either $B(B_H D \arrows BV)$ or as $B(B_V D \arrows BH)$.

\begin{example}Let $\Gamma$ be a Lie group acting on a Lie groupoid $G\rightrightarrows
M$ by groupoid automorphisms. Consider the double Lie groupoid
$G\rtimes \Gamma$,
\begin{equation}\label{eqn:actiondouble}
    \doublegroupoid{G}{G\rtimes \Gamma}{M}{M\rtimes \Gamma},
\end{equation}
where the groupoid structure on $G \rtimes \Gamma \arrows M \rtimes \Gamma$ is the Cartesian product of $G \arrows M$ and the trivial Lie groupoid $\Gamma \arrows \Gamma$.

If we apply the usual classifying space functor in the vertical direction, then we get the topological groupoid $BG \rtimes \Gamma \arrows BG$ corresponding to the induced action of $\Gamma$ on $BG$. The classifying space of this action groupoid (hence the classifying space of \eqref{eqn:actiondouble}) is the homotopy quotient
\begin{equation}\label{eqn:classaction}
  (BG\times E\Gamma)/\Gamma.
\end{equation}
On the other hand, we may first apply the usual classifying space functor in the horizontal direction, obtaining the homotopy quotient groupoid
\begin{equation}\label{eqn:homquot}
(G \times E\Gamma)/\Gamma \arrows (M \times E\Gamma)/\Gamma.
\end{equation}
By the Eilenberg-Zilber theorem, the classifying space of \eqref{eqn:homquot} is homeomorphic to \eqref{eqn:classaction}. Thus we may say that the classifying space of the homotopy quotient equals the homotopy quotient of the classifying space.

Now suppose that $\Gamma$ is a \emph{discrete} group acting on a manifold $M$. The action of $\Gamma$ naturally lifts to an action of $\Gamma$ on the fundamental groupoid $\Pi_1(M)\rightrightarrows M$. In this case, the ``action double groupoid'' \eqref{eqn:actiondouble} can be identified with the fundamental double groupoid
\begin{equation*}
     \doublegroupoid{\Pi_1(M)}{\Pi_1(M \rtimes \Gamma)}{M}{M \rtimes \Gamma}.
\end{equation*}
According to the above discussion, the classifying space\footnote{One should
compare this computation with the definition of fundamental group of
a differentiable stack in \cite{noohi} and \cite{laurent-xu-tu}.} $B(\Pi_1(M \rtimes \Gamma))$ is
$(B\Pi_1(M)\times E\Gamma)/\Gamma$. When $M$ is connected, it follows from the associated homotopy long exact sequence that $B(\Pi_1(M \rtimes \Gamma))$ is the Eilenberg-Maclane space $K(G,1)$, where $G$ is a (nonabelian) extension of $\Gamma$ by the fundamental group $\pi_1(M)$.
\end{example}

\section{Lie $n$-groupoids and local Lie $n$-groupoids}\label{sec:n}

\subsection{Definitions}

We now review the ``simplicial'' definitions of (weak) Lie $n$-groupoids and local Lie $n$-groupoids. The simplicial approach to $n$-groupoids goes back to Duskin \cite{duskin} in the discrete case, and the smooth analogue appeared in \cite{henriques:integrating}. The definition of local Lie $n$-groupoid is due to Zhu \cite{zhu:kan}.

Let $X = \{X_q\}$ be a simplicial manifold. For $q \geq 1$ and $0 \leq k \leq q$, a $(q,k)$-\emph{horn} of $X$ consists of a $q$-tuple $(x_0, \dots, x_{k-1}, x_{k+1}, \dots x_q)$, where $x_i \in X_{q-1}$, satisfying the \emph{horn compatibility equations}
\begin{equation}\label{eqn:horncompat}
     f_i^{q-1} x_j = f_{j-1}^{q-1} x_i
\end{equation}
for $i < j$. The space $\Lambda_{q,k} X$ of $(q,k)$-horns is smooth since the face maps $f_i^q$ are submersions.

There are natural \emph{horn maps} $\lambda_{q,k}: X_q \to \Lambda_{q,k} X$, where
\begin{equation*}
 \lambda_{q,k}(x) = (f_0^q x, \dots, \widehat{f_k^q x}, \dots f_q^q x)
\end{equation*}
 for $x \in X_q$. It is immediate from \eqref{eqn:twoface} that $\lambda_{q,k}(x)$ satisfies the horn compatibility equations \eqref{eqn:horncompat}; in fact, the purpose of the horn compatibility equations is to axiomatize the properties satisfied by $\lambda_{q,k}(x)$.

\begin{definition}\label{dfn:n}
     A \emph{Lie $n$-groupoid} is a simplicial manifold such that the horn maps $\lambda_{q,k}$ are
\begin{enumerate}
     \item surjective submersions for all $q \geq 1$, and
     \item diffeomorphisms for all $q > n$.
\end{enumerate}
\end{definition}

\begin{definition}\label{dfn:localn}
     A \emph{local Lie $n$-groupoid} is a simplicial manifold such that the horn maps $\lambda_{q,k}$ are
\begin{enumerate}
     \item submersions for all $q \geq 1$, and
     \item injective \'{e}tale for all $q > n$.
\end{enumerate}
\end{definition}

We remark that in Definitions \ref{dfn:n} and \ref{dfn:localn}, the case $q=1$ is redundant. Specifically, the requirement that $\lambda_{1,0} = f_1^1$ and $\lambda_{1,1} = f_0^1$ be surjective submersions is already part of the definition of simplicial manifold.

\subsection{Applying the bar functor to the nerve}

Let $D$ be a double Lie groupoid as in \eqref{eqn:dbliegpd}. In
\S\ref{sec:nerve} we described the nerve construction, which
produces a bisimplicial manifold $ND$. Now we shall apply the bar
construction, described in \S\ref{sec:bar}, to obtain a simplicial
manifold $\barW ND$.

We begin by describing $\barW_r ND$ for the first few values of $r$. From \eqref{eqn:bar}, we have the following:
\begin{align*}
     \barW_0 ND &= N_{0,0} D = M, \\
\barW_1 ND &= N_{0,1}D\bitimes{v_0}{h_1} N_{1,0}D = V \bitimes{s}{t} H,\\
\barW_2 ND &= N_{0,2}D \bitimes{v_0}{h_1} N_{1,1}D \bitimes{v_0}{h_2} N_{2,0}D = V^{(2)} \bitimes{p_2}{t_V} D \bitimes{s_H}{p_1} H^{(2)},
\end{align*}
where $p_1$ and $p_2$ are, respectively, the ``projection onto the first and second component'' maps. Eliminating redundancies arising from the fiber products, we may simplify the description of $\barW_2 ND$  to $V \bitimes{s}{t^2} D \bitimes{s^2}{t} H$, where $s^2 \defequal s \circ s_V = s \circ s_H$ and $t^2 \defequal t \circ t_V = t \circ t_H$. So, up to $r=2$, the simplicial manifold $\barW ND$ is of the form
\begin{equation}
         \threenerve{M}{V \bitimes{s}{t} H}{V \bitimes{s}{t^2} D \bitimes{s^2}{t} H}.
\end{equation}
From \eqref{eqn:barface}, we have the following equations describing the face maps:
\begin{align*}
     f_0^1 (\theta, \eta) &= s(\eta), \\
f_1^1(\theta,\eta) &= t(\theta),
\end{align*}
for $(\theta, \eta) \in V \bitimes{s}{t} H$, and
\begin{align}
     f_0^2 (\theta, \alpha, \eta) &= (s_V \alpha, \eta),\label{eqn:face20} \\
     f_1^2 (\theta, \alpha, \eta) &= (\theta \cdot t_V \alpha, s_H \alpha \cdot \eta), \\
     f_2^2 (\theta, \alpha, \eta) &= (\theta, t_H \alpha),\label{eqn:face22}
\end{align}
for $(\theta, \alpha, \eta) \in V \bitimes{s}{t^2} D \bitimes{s^2}{t} H$.

A $1$-simplex $(\theta, \eta)$ and a $2$-simplex $(\theta, \alpha, \eta)$ can be drawn as follows:
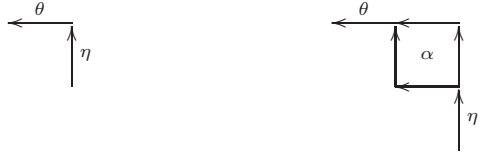
\begin{align*}
\xymatrix@M=0pt{ & \ar_{\theta}[l] \\ & \ar_{\eta}[u]}  &&
\xymatrix@M=0pt{ & \ar_{\theta}[l] & \ar[l] \\ & \ar[u] & \ar[u]
\ar[l] \ar@{}[ul]|{\alpha} \\ & & \ar_{\eta}[u]}.
\end{align*}
We emphasize that, in these diagrams, elements of $V$ are drawn as horizontal lines, and elements of $H$ are drawn as vertical lines. The face maps $f_0^2$, $f_1^2$, and $f_2^2$ are respectively illustrated as follows:
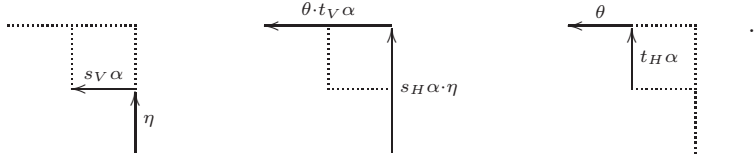
\begin{align*}
     \xymatrix@M=0pt{ & \ar@{.}[l] & \ar@{.}[l] \\ & \ar@{.}[u] & \ar@{.}[u] \ar_{s_V \alpha}[l]  \\ & & \ar_{\eta}[u]} &&
     \xymatrix@M=0pt{ &  & \ar_{\theta \cdot t_V \alpha}[ll] \\ & \ar@{.}[u] &  \ar@{.}[l]  \\ & & \ar_{s_H \alpha \cdot \eta}[uu]} &&
     \xymatrix@M=0pt{ & \ar_{\theta}[l] & \ar@{.}[l] \\ & \ar_{t_H \alpha}[u] & \ar@{.}[u] \ar@{.}[l]  \\ & &
     \ar@{.}[u]}\qquad .
\end{align*}
In general, a $r$-simplex in $\barW ND$ consists of an element $\theta \in V$, a triangular array $\{ \alpha_{ij} \in D \}$ for $1 \leq i \leq j \leq r$, and an element $\eta \in H$, satisfying the compatibility equations \eqref{eqn:ncompat} and
\begin{align*}
     s(\theta) &= t^2(\alpha_{11}), & t(\eta) &= s^2(\alpha_{rr}).
\end{align*}
The pictures of a $3$-simplex and a $4$-simplex are as follows:
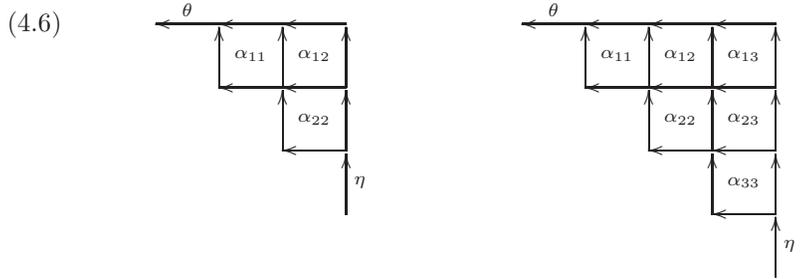
\begin{align}\label{eqn:34simplex}
     \xymatrix@M=0pt{ & \ar_{\theta}[l] & \ar[l] & \ar[l] \\ & \ar[u] & \ar[u] \ar[l] \ar@{}[ul]|{\alpha_{11}} & \ar[u] \ar[l] \ar@{}[ul]|{\alpha_{12}} \\ & & \ar[u] & \ar[u] \ar[l] \ar@{}[ul]|{\alpha_{22}} \\ & & & \ar_{\eta}[u]} &&
     \xymatrix@M=0pt{ & \ar_{\theta}[l] & \ar[l] & \ar[l] & \ar[l] \\ & \ar[u] & \ar[u] \ar[l] \ar@{}[ul]|{\alpha_{11}} & \ar[u] \ar[l] \ar@{}[ul]|{\alpha_{12}} & \ar[u] \ar[l] \ar@{}[ul]|{\alpha_{13}} \\ & & \ar[u] & \ar[u] \ar[l] \ar@{}[ul]|{\alpha_{22}} & \ar[u] \ar[l] \ar@{}[ul]|{\alpha_{23}}\\ & & & \ar[u] & \ar[u] \ar[l] \ar@{}[ul]|{\alpha_{33}} \\ & & & &
     \ar_{\eta}[u]}.
\end{align}
There is a nice description of the face maps in terms of pictures such as \eqref{eqn:34simplex}, where the face map $f_i^r : \barW_r ND \to \barW_{r-1} ND$ is given by deleting all the arrows in the $i$th row or column (Note that the columns and rows are counted starting from $0$ in the upper left corner). If an arrow forming the shared boundary of two squares is deleted, then those squares are multiplied across that boundary. If an arrow forming part of an outer boundary of a square is deleted, then that square (as well as any arrows touching the same outer boundary) is also removed. For example, the face maps $f_0^4$, $f_1^4$, and $f_2^4$ are respectively given by
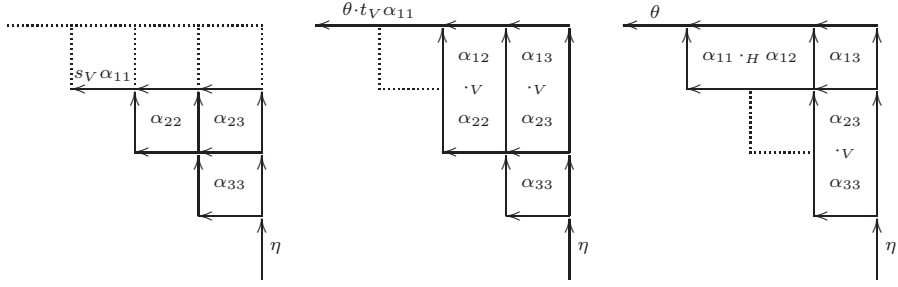
\begin{align*}
          \xymatrix@M=0pt{ & \ar@{.}[l] & \ar@{.}[l] & \ar@{.}[l] & \ar@{.}[l] \\ & \ar@{.}[u] & \ar@{.}[u] \ar_{s_V \alpha_{11}}[l]  & \ar@{.}[u] \ar[l] & \ar@{.}[u] \ar[l] \\ & & \ar[u] & \ar[u] \ar[l] \ar@{}[ul]|{\alpha_{22}} & \ar[u] \ar[l] \ar@{}[ul]|{\alpha_{23}}\\ & & & \ar[u] & \ar[u] \ar[l] \ar@{}[ul]|{\alpha_{33}} \\ & & & & \ar_{\eta}[u]}
&&
     \xymatrix@M=0pt{ &  & \ar_{\theta \cdot t_V \alpha_{11}}[ll] & \ar[l] & \ar[l] \\ & \ar@{.}[u] &  \ar@{.}[l] &  \ar@{}[l]|{\cdot_V} \ar@{}[ul]|{\alpha_{12}} &  \ar@{}[l]|{\cdot_V} \ar@{}[ul]|{\alpha_{13}} \\ & & \ar[uu] & \ar[uu] \ar[l] \ar@{}[ul]|{\alpha_{22}} & \ar[uu] \ar[l] \ar@{}[ul]|{\alpha_{23}}\\ & & & \ar[u] & \ar[u] \ar[l] \ar@{}[ul]|{\alpha_{33}} \\ & & & & \ar_{\eta}[u]}
&&
     \xymatrix@M=0pt{ & \ar_{\theta}[l] & & \ar[ll] & \ar[l] \\ & \ar[u] & \ar@{}[u]|{\cdot_H} \ar@{}[ul]|{\alpha_{11}} & \ar[u] \ar[ll] \ar@{}[ul]|{\alpha_{12}} & \ar[u] \ar[l] \ar@{}[ul]|{\alpha_{13}} \\ & & \ar@{.}[u] &  \ar@{.}[l] & \ar@{}[l]|{\cdot_V} \ar@{}[ul]|{\alpha_{23}}\\ & & & \ar[uu] & \ar[uu] \ar[l] \ar@{}[ul]|{\alpha_{33}} \\ & & & &
     \ar_{\eta}[u]}.
\end{align*}

\begin{lemma}\label{lemma:wnd2}
     For the simplicial manifold $\barW ND$, the $2$-horn maps $\lambda_{2,k}$ are submersions. Furthermore, if $D$ is a full double Lie groupoid, then the $2$-horn maps are surjective.
\end{lemma}

\begin{proof}
The general idea for the proof is to express each horn map as a composition of maps, where one factor is the fiber product of identity maps with the double-source map, and the other factors are diffeomorphisms. We will do this for $k=2$ and leave the other two cases for the reader.

The horn map $\lambda_{2,2}$ factors as the composition of the following three maps:
\begin{align*}
\barW_2 ND = V \bitimes{s}{t^2} D \bitimes{s^2}{t} H &\to V \bitimes{s}{s\circ t_V} D \bitimes{s^2}{t} H \\
(\theta, \alpha, \eta) &\mapsto (\theta \cdot t_V \alpha, \alpha, \eta), \\
V \bitimes{s}{s\circ t_V} D \bitimes{s^2}{t} H &\to V \bitimes{s}{t} H \bitimes{s}{s} V \bitimes {s}{t} H \\
(\theta, \alpha, \eta) &\mapsto (\theta, s_H \alpha, s_V \alpha, \eta), \\
V \bitimes{s}{t} H \bitimes{s}{s} V \bitimes {s}{t} H &\to \Lambda_{2,2} \barW ND \\
(\theta_1, \eta_1, \theta_2, \eta_2) &\mapsto \left( (\theta_2, \eta_2), (\theta_1, \eta_1 \cdot \eta_2) \right).
\end{align*}
The composition of these maps is described by the following diagram:
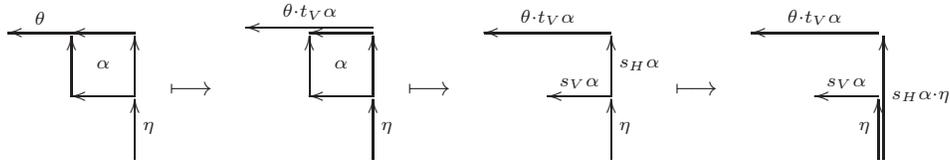
\begin{align*}
\xymatrix@M=0pt{ & \ar_{\theta}[l] & \ar[l] \\ & \ar[u] & \ar[u]
\ar[l] \ar@{}[ul]|{\alpha} \\ & & \ar_{\eta}[u]} &\mapstoxy&
\xymatrix@M=0pt{ &  & \ar[l] \ar@<-2pt>_{\theta \cdot t_V
\alpha}[ll] \\ & \ar[u] & \ar[u] \ar[l] \ar@{}[ul]|{\alpha} \\ & &
\ar_{\eta}[u]} &\mapstoxy& \xymatrix@M=0pt{ &  & \ar_{\theta \cdot
t_V \alpha}[ll] \\ &  & \ar_{s_H \alpha}[u] \ar_{s_V \alpha}[l] \\ &
& \ar_{\eta}[u]} &\mapstoxy& \xymatrix@M=0pt{ &  & \ar_{\theta \cdot
t_V \alpha}[ll] \\ &  & \ar_{s_V \alpha}[l] \\ & & \ar^{\eta}[u]
\ar@<-2pt>_{s_H \alpha \cdot \eta}[uu]}.
\end{align*}
The first and third maps are diffeomorphisms, and the second map is $\left(\id, (s_H, s_V), \id \right)$. Since the double-source map $(s_H, s_V)$ is a submersion, it follows that $\lambda_{2,2}$ is a submersion. If the double Lie groupoid is full (so that $(s_H, s_V)$ is also surjective), then $\lambda_{2,2}$ is surjective.
\end{proof}

\begin{lemma}\label{lemma:wndgeq3}
     For the simplicial manifold $\barW ND$, the $r$-horn maps $\lambda_{r,k}$ are diffeomorphisms for $r > 2$.
\end{lemma}
\begin{proof}
For $r > 2$, let $\Upsilon_r$ be the ``partial horn space'' consisting of triplets $(x_0, x_1, x_2)$ of $(r-1)$-simplices $x_i \in \barW_{r-1} ND$, satisfying the following compatibility equations:
\begin{align}\label{eqn:partialhorncompat}
     f_1^{r-1} x_0 &= f_0^{r-1}x_2, & f_0^{r-1} x_0 &= f_0^{r-1}x_1, &  f_1^{r-1} x_1 &= f_1^{r-1}x_2.
\end{align}
These equations are precisely the horn compatibility conditions that involve faces $0$, $1$, and $2$.

There is a natural map $\upsilon_r : \barW_r ND \to \Upsilon_r$, given by
\begin{equation*}
     \upsilon_r (x) = (f_0^r x, f_1^r x, f_2^r x).
\end{equation*}
When $k>2$, this map factors through the horn space $\Lambda_{r,k} \barW ND$, so that we have the following commutative diagram:
\begin{equation}\label{eqn:upsilondiagram}
     \xymatrix{ \barW_r ND \ar^-{\lambda_{r,k}}[r] \ar^{\upsilon_r}[dr] & \Lambda_{r,k} \barW ND \ar[d] \\ & \Upsilon_r}.
\end{equation}
We will construct an inverse of the map $\upsilon_r$. Consider an arbitrary $(x_0, x_1, x_2) \in \Upsilon_r$. Let us depict $x_0 \in \barW_{r-1} ND$ as
\begin{equation}\label{eqn:x0}
     \xymatrix@M=0pt{ & \ar_{\xi}[l] & \ar[l] &  & \ar[l] \\ & \ar[u] & \ar[u] \ar[l] \ar@{}[ul]|{\alpha_{22}} & \ar[u] \ar@{}[ul]|{\cdots} & \ar[u] \ar[l] \ar@{}[ul]|{\alpha_{2r}} \\ & & &  \ar@{}[ul]|{\ddots} & \ar[l] \ar@{}[ul]|{\vdots}\\ & & & \ar[u] & \ar[u] \ar[l] \ar@{}[ul]|{\alpha_{rr}} \\ & & & &
     \ar_{\eta}[u]}.
\end{equation}
The first of the compatibility conditions \eqref{eqn:partialhorncompat} says that the first two rows of $x_2$ are of the form
\begin{equation}\label{eqn:x2}
       \xymatrix@M=0pt@R+1pc@C+1pc{ & \ar_{\theta}[l] & \ar[l] & \ar[l] &  & \ar[l] \\
& \ar[u] & \ar[u] \ar[l] \ar@{}[ul]|{\gamma} & \ar[u] \ar[l] \ar@{}[ul]|{\alpha_{13}} & \ar[u] \ar@{}[ul]|{\cdots} & \ar[u] \ar[l] \ar@{}[ul]|{\alpha_{1r}} \\
& & \ar[u] & \ar[u] \ar[l] \ar@{}[ul]|{\alpha_{23} \cdot_V \alpha_{33}} & \ar[u] \ar@{}[ul]|{\cdots} & \ar[u] \ar[l] \ar@{}[ul]|{\alpha_{2r} \cdot_V \alpha_{3r}}}
\end{equation}
where $s_V(\gamma) = \xi \cdot t_V \alpha_{22}$, and that the remaining rows are the same as those of $x_0$. The latter two of the compatibility conditions \eqref{eqn:partialhorncompat} say that the first row of $x_1$ is of the form
\begin{equation*}
     \xymatrix@M=0pt@R+1pc@C+1pc{ & \ar_{\xi'}[l] & \ar[l] & \ar[l] &  & \ar[l] \\
& \ar[u] & \ar[u] \ar[l] \ar@{}[ul]|{\beta} & \ar[u] \ar[l] \ar@{}[ul]|{\alpha_{13} \cdot_V \alpha_{23}} & \ar[u] \ar@{}[ul]|{\cdots} & \ar[u] \ar[l] \ar@{}[ul]|{\alpha_{1r} \cdot_V \alpha_{2r}}}
\end{equation*}
where $s_V (\beta) = s_V (\alpha_{22})$ and $\xi' \cdot t_V \beta = \theta \cdot t_V \gamma$, and that the remaining rows are the same as those of $x_0$.

We then construct an $r$-simplex $x$ as
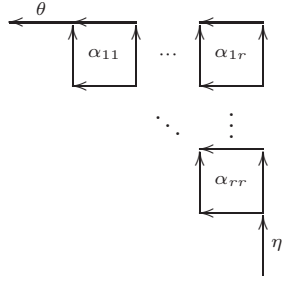
\begin{equation*}
     \xymatrix@M=0pt{ & \ar_{\theta}[l] & \ar[l] &  & \ar[l] \\ & \ar[u] & \ar[u] \ar[l] \ar@{}[ul]|{\alpha_{11}} & \ar[u] \ar@{}[ul]|{\cdots} & \ar[u] \ar[l] \ar@{}[ul]|{\alpha_{1r}} \\ & & &  \ar@{}[ul]|{\ddots} & \ar[l] \ar@{}[ul]|{\vdots}\\ & & & \ar[u] & \ar[u] \ar[l] \ar@{}[ul]|{\alpha_{rr}} \\ & & & & \ar_{\eta}[u]}
\end{equation*}
where $\theta \in V$, $\eta \in H$, and $\alpha_{ij} \in D$ (with the exception of $\alpha_{11}$ and $\alpha_{12}$) are the same as those in \eqref{eqn:x0}--\eqref{eqn:x2}, and
\begin{align*} \alpha_{12} &= \beta \cdot_V \iota_V(\alpha_{22}), & \alpha_{11} = \gamma \cdot_H \iota_H(\alpha_{12}).\end{align*}
Here, $\iota_V$ and $\iota_H$ denote, respectively, the vertical and horizontal inverse maps on $D$.

We can directly see from the definition of the face maps that the map $(x_0, x_1, x_2) \mapsto x$ given by the above construction is inverse to $\upsilon_r$, and therefore $\upsilon_r$ is a diffeomorphism. One can similarly show that for $k > 2$, there is a unique $(r-1)$-simplex $x_k$ that is compatible with $x_0$, $x_1$, and $x_2$, and therefore the map from $\Lambda_{r,k} \barW ND$ to $\Upsilon_r$ given by forgetting all but the $0$th, $1$st, and $2$nd faces is a diffeomorphism. Thus we conclude (c.f.\ diagram \eqref{eqn:upsilondiagram}) that the horn map $\lambda_{r,k}$ is a diffeomorphism for $k > 2$.

Using the same argument as above with a partial horn space associated to the last three faces, one can prove that $\lambda_{r,k}$ is a diffeomorphism for $k < r - 2$. The cases not covered under $k>2$ or $k<r-2$ are $\lambda_{3,1}$, $\lambda_{3,2}$ and $\lambda_{4,2}$, which the reader may check directly.
\end{proof}

An immediate consequence of Lemmas \ref{lemma:wnd2} and \ref{lemma:wndgeq3} is the following:
\begin{thm}
 Let $D$ be a double Lie groupoid. Then $\barW ND$ is a local Lie $2$-groupoid. Furthermore, if $D$ is a full double Lie groupoid, then $\barW ND$ is a Lie $2$-groupoid.
\end{thm}

\begin{remark}
     We observe that the result of Lemma \ref{lemma:wndgeq3} is stronger than condition (2) of Definition \ref{dfn:localn}. This perhaps suggests that the definition of local Lie $n$-groupoid should be strengthened.
\end{remark}

\section{From local $2$-groupoids to groupoids}\label{sec:groupoidization}

\subsection{The functor $\mathfrak{G}$}

Let $X_\bullet$ be a simplicial space. We may form a topological groupoid $\mathfrak{G}(X) \arrows X_0$ by taking $\mathfrak{G}(X)$ to be the groupoid generated by $X_1$ with relations generated by the $2$-simplices $X_2$. More precisely, $\mathfrak{G}(X)$ consists of formal products $y_1^{\pm 1} \cdots y_m^{\pm 1}$ of composable elements $y_i \in X_1$ modulo relations generated by those of the form
\begin{equation}\label{eqn:grelations}
   f_1^2 z \sim f_2^2 z \cdot f_0^2 z
\end{equation}
for $z \in X_2$, in addition to the ``tautological relations'' $y \cdot y^{-1} \sim \delta_0^0 f_1^1 y$ and $y^{-1} \cdot y \sim \delta_0^0 f_0^1 y$. We note that the degeneracy map $\delta_0^0: X_0 \to X_1$ automatically satisfies the axioms of a unit map as a result of \eqref{eqn:facedeg}.

\begin{remark}\label{rmk:truncation}
     If $X_\bullet$ is a $2$-groupoid, then the surjectivity of the horn maps $\lambda_{2,k}$ implies that any formal product $y_1^{\pm 1} \cdots y_m^{\pm 1}$ can be (nonuniquely) reduced to a single element $y \in X_1$. We thereby recover the well-known ``truncation'' process, where $\mathfrak{G}(X) = X_1/\sim$, with the relation generated by $f_k^2 z \sim f_k^2 z'$ for any $z, z' \in X_2$ such that $\lambda_{2,k} z = \lambda_{2,k} z'$. In the general case, the functor $X \mapsto \mathfrak{G}(X)$ coincides with the composition of truncation with Kan extension \cite{zhu:kan}.
\end{remark}

We now apply the functor $\mathfrak{G}$ to $\barW ND$, where $D$ is a double Lie groupoid as in \eqref{eqn:dbliegpd}.  We will abbreviate $\mathfrak{G}(\barW ND)$ to simply $\mathfrak{G}(D)$. In this case, the main simplification is given by the following lemma.

\begin{lemma}\label{lemma:gdinv}
     Let $(\theta, \eta) \in V \bitimes{s}{t} H = \barW_1 ND$. In $\mathfrak{G}(D)$ the following relation holds:
\begin{equation}\label{eqn:gdinv}
     (\theta,\eta)^{-1} \sim (1,\eta^{-1}) \cdot (\theta^{-1},1).
\end{equation}
\end{lemma}
\begin{proof}
     Note that in \eqref{eqn:gdinv} we use ``$1$'' as a generic symbol for unit elements in $V$ and $H$.

We use relations from the $2$-simplices $z \defequal (\theta, 1(\eta), \eta^{-1})$ and $z' \defequal (\eta, 1(\eta^{-1}), 1)$, pictured below:
\begin{align*}
   z = \;\raisebox{1pc}{$\xymatrix@M=0pt{ & \ar_{\theta}[l] & \ar_{1}[l] \\ & \ar^{\eta}[u] & \ar_{\eta}[u] \ar^{1}[l] \ar@{}[ul]|{1(\eta)} \\ & & \ar_{\eta^{-1}}[u]}$} &&
z' = \;\raisebox{1pc}{$\xymatrix@M=0pt{ & \ar_{\theta}[l] & \ar_{\theta^{-1}}[l] \\ & \ar^{1}[u] & \ar_{1}[u] \ar^{\theta^{-1}}[l] \ar@{}[ul]|{1(\theta^{-1})} \\ & & \ar_{1}[u]}$}.
\end{align*}
For these $2$-simplices, the relations \eqref{eqn:grelations} are respectively
\begin{align}
    (\theta \cdot 1, \eta \cdot \eta^{-1}) = (\theta, 1) &\sim (\theta, \eta) \cdot (1, \eta^{-1}), \label{eqn:hvinv1} \\
(\theta \cdot \theta^{-1}, 1 \cdot 1) = (1, 1) &\sim (\theta, 1) \cdot (\theta^{-1}, 1). \label{eqn:hvinv2}
\end{align}
Combining \eqref{eqn:hvinv1} and \eqref{eqn:hvinv2}, we have
\begin{equation*}
     (1,1) \sim (\theta, \eta) \cdot (1, \eta^{-1}) \cdot (\theta^{-1},1),
\end{equation*}
which implies the statement of the lemma.
\end{proof}

As a result of Lemma \ref{lemma:gdinv}, it is not necessary to include formal inverses as generators for $\mathfrak{G}(D)$. Thus we have the following description:
\begin{definition}\label{dfn:groupoidification}
     The groupoid $\mathfrak{G}(D) \arrows M$ associated to a double Lie groupoid \eqref{eqn:dbliegpd} consists of formal products $(\theta_1, \eta_1) \cdots (\theta_m, \eta_m)$ of composable elements $\theta_i \in V$, $\eta_i \in H$, modulo relations generated by those of the form
\begin{equation}\label{eqn:gdrelations}
   (\theta \cdot t_V \alpha, s_H \alpha \cdot \eta) \sim (\theta, t_H \alpha) \cdot (s_V \alpha, \eta)
\end{equation}

for $(\theta, \alpha, \eta) \in V \bitimes{s}{t^2} D \bitimes{s^2}{t} H$.
\end{definition}

\begin{remark}\label{rmk:fullgdtrunc}
     In the case where $D$ is full, following Remark \ref{rmk:truncation}, we have that $\mathfrak{G}(D)$ is simply $(V \bitimes{s}{t} H)/\sim$, with the following relation:
\begin{equation}\label{eqn:gdfullrelation}
(\theta \cdot t_V \alpha, s_H \alpha \cdot \eta) \sim (\theta \cdot t_V \alpha', s_H \alpha' \cdot \eta)
\end{equation}
for $\alpha, \alpha' \in D$ such that $t_H \alpha = t_H \alpha'$ and $s_V \alpha = s_V \alpha'$. This relation guarantees that the multiplication given by \eqref{eqn:gdrelations} is well-defined.

In the full case, the isotropy groups of $\mathfrak{G}(D)$ coincide with the ``fundamental groups'' of \cite{ce-he-me:double}. Note that we have arrived at the relation \eqref{eqn:gdfullrelation} and the multiplication \eqref{eqn:gdrelations} by specializing a more general, natural construction.
\end{remark}

\subsection{Fundamental double groupoids and Haefliger fundamental groupoids}
\label{sec:fundamental}

Consider $\mathfrak{G}(\Pi_1(G))$, where $\Pi_1(G)$ is the
fundamental double groupoid of a Lie groupoid $G \arrows M$ whose
source map is a fibration (see Example \ref{example:fundamental}).
In this case, we recover a simple description of the Haefliger
fundamental groupoid \cite{haefliger:homotopy, haefliger:orbi,
mo-mr:fundamental} of $G$ as $(G \bitimes{s}{t} \Pi_1(M))/\sim$,
with the relation \eqref{eqn:gdfullrelation}. For any
$m\in M$, the isotropy group at $m$ of the Haefliger fundamental groupoid is isomorphic to the fundamental group of the classifying space $BG$ at $m$.

In the case where the source map of $G$ is not a fibration, then $\Pi_1(G) \arrows \Pi_1(M)$ is not a Lie groupoid, but it is in fact a local Lie ($1$)-groupoid in the sense of Definition \ref{dfn:localn} (the associated simplicial manifold is the fundamental groupoid of the nerve of $G$). Nonetheless, our construction essentially goes through verbatim. The local Lie groupoid structure of $\Pi_1(G) \arrows \Pi_1(M)$ is compatible with the groupoid structure of $\Pi_1(G) \arrows G$, so that we might call \eqref{eqn:fundamental} a \emph{semilocal double Lie groupoid}. In particular, the double-nerve $N_{\bullet,\bullet} \Pi_1(G)$ is still defined and is a bisimplicial manifold, and when we apply the bar construction to $N_{\bullet,\bullet} \Pi_1(G)$, we obtain a local Lie $2$-groupoid. Applying the functor $\mathfrak{G}$, we again obtain the Haefliger fundamental groupoid of $G$.

We point out that the ``equivalences'' and ``deformations'' in the definition of the Haefliger fundamental groupoid are both included in the equivalence relation \eqref{eqn:gdrelations}.

\section{Toward symplectic 2-groupoids}\label{sec:symplectic}

Recall from Example \ref{example:symplectic} that a symplectic double groupoid is a double Lie groupoid as in \eqref{eqn:dbliegpd}, where $D$ is equipped with a symplectic structure such that both
$D\rightrightarrows V$ and $D\rightrightarrows H$ are symplectic
groupoids.  In this section, we study the local Lie $2$-groupoid
$\barW ND$ associated to a symplectic double groupoid $D$. Particularly, we wish to investigate the geometric structure on $\barW ND$ that arises from the symplectic structure on $D$.

Although there is not yet an established definition of a ``symplectic $2$-groupoid'', it is expected \cite{severa:sometitle} that symplectic $2$-groupoids should be the objects that integrate Courant algebroids. We argue, based on the following diagram, that (local) $2$-groupoids of the form $\barW ND$, where $D$ is a symplectic double groupoid, are examples of symplectic $2$-groupoids integrating Courant algebroids of the form $A \oplus A^*$, where $(A, A^*)$ is a Lie bialgebroid \cite{lwx}.
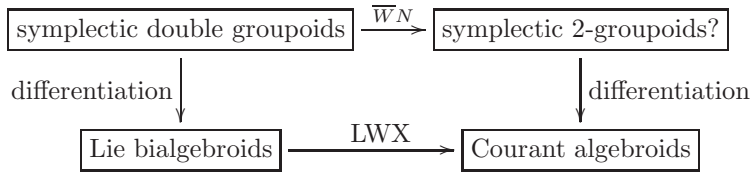
\begin{equation*}
     \xymatrix{ \fbox{symplectic double groupoids} \ar_{\mbox{differentiation}}[d] \ar^-{\barW N}[r] & \fbox{symplectic $2$-groupoids?} \ar^{\mbox{differentiation}}[d] \\ \fbox{Lie bialgebroids} \ar^{\mbox{LWX}}[r] & \fbox{Courant algebroids}}
\end{equation*}

We observe that there is a natural map $\nu$ from $\barW_2 ND = V \bitimes{s}{t^2} D \bitimes{s^2}{t} H$ to $D$ given by $\nu(\theta,\alpha,\eta) = \alpha$. The symplectic
$2$-form $\omega$ on $D$ pulls back to define a closed $2$-form $\Omega \defequal \nu^* \omega$
on $\barW_2 ND$. In the remainder of this section, we consider some properties satisfied by $\Omega$ that might be considered ``higher'' analogues of the multiplicativity and nondegeneracy conditions satisfied by $2$-forms in the case of symplectic ($1$)-groupoids. This leads us to a tentative proposal for a definition of symplectic $2$-groupoid in \S\ref{sec:sympdef}.

\subsection{Multiplicativity}\label{subec:multiplicativity}
Let $X_\bullet$ be a simplicial manifold. We say that a differential form $\alpha$
on $X_q$ is \emph{multiplicative} if
\[
\sum_{i} (-1)^i(f_i^{q+1})^*\alpha=0,
\]
i.e.\ if $\alpha$ is closed with respect to the simplicial coboundary.

\begin{definition}\label{dfn:presymplectic}
     A \emph{presymplectic double groupoid} is a double Lie groupoid $D$ equipped with a closed $2$-form $\omega$ that is multiplicative with respect to both Lie groupoid structures.
\end{definition}

\begin{prop}\label{prop:mult-form}
If $D$ is a presymplectic double groupoid, then the closed $2$-form $\Omega = \nu^* \omega$ on $\barW_2 ND$ is multiplicative.
\end{prop}
\begin{proof}
Recall that $\barW_3 ND$ is equal to $V \bitimes{s}{t^2} D \bitimes{s_H}{t_H} D \bitimes{s_V}{t_V} D \bitimes{s^2}{t} H$, which consists of $(\theta, \alpha_1, \alpha_2, \alpha_3, \eta) \in V \times D \times D \times D \times H$ such that
\begin{align*}
   s(\theta) &= t^2(\alpha_1), &  s_H(\alpha_1) &= t_H(\alpha_2), & s_V(\alpha_2) &=t_V(\alpha_3), &
s^2(\alpha_3)&=t(\eta).
\end{align*}
The face maps from $\barW_3 ND$ to $\barW_2 ND$ are as follows:
\begin{align*}
    f^3_0(\theta, \alpha_1, \alpha_2, \alpha_3, \eta) &= (s_V \alpha_1, \alpha_3, \eta), \\
   f^3_1(\theta, \alpha_1, \alpha_2, \alpha_3, \eta) &= (\theta \cdot t_V \alpha_1, \alpha_2 \cdot_V \alpha_3, \eta), \\
   f^3_2(\theta, \alpha_1, \alpha_2, \alpha_3, \eta) &= (\theta, \alpha_1 \cdot_H \alpha_2, s_H \alpha_3 \cdot \eta), \\
   f^3_3(\theta, \alpha_1, \alpha_2, \alpha_3, \eta) &= (\theta, \alpha_1, t_H \alpha_3).
\end{align*}
Composing these face maps with the projection $\nu$, we have
\begin{align*}
   \nu\circ f^3_0(\theta, \alpha_1, \alpha_2, \alpha_3, \eta) &= \alpha_3, \\
   \nu\circ f^3_1(\theta, \alpha_1, \alpha_2, \alpha_3, \eta) &= \alpha_2 \cdot_V \alpha_3, \\
   \nu\circ f^3_2(\theta, \alpha_1, \alpha_2, \alpha_3, \eta) &= \alpha_1 \cdot_H \alpha_2, \\
   \nu\circ f^3_3(\theta, \alpha_1, \alpha_2, \alpha_3, \eta) &= \alpha_1.
\end{align*}

Define $\mu_2: \barW_3 ND \to D$ by $\mu_2(\theta, \alpha_1, \alpha_2, \alpha_3, \eta)=\alpha_2$. As $\omega$ is a multiplicative
$2$-form on $D\rightrightarrows V$, we have that
\begin{equation}\label{eqn:omega-V}
(\nu\circ f^3_0)^*\omega+\mu_2^*\omega=(\nu\circ f^3_1)^*\omega.
\end{equation}
Similarly, as $\omega$ is a multiplicative $2$-form on
$D\rightrightarrows H$, we have that
\begin{equation}\label{eqn:omega-H}
(\nu\circ f^3_3)^*\omega+\mu_2^*\omega=(\nu \circ f^3_2)^*\omega.
\end{equation}
From equations (\ref{eqn:omega-V})--(\ref{eqn:omega-H}), we see that
the simplicial coboundary of $\Omega$ is
\begin{equation*}
   \sum_i (-1)^i(f^3_i)^*\Omega= \sum_i (-1)^i(\nu \circ f^3_i)^*\omega = 0.  \qedhere
\end{equation*}
\end{proof}

%\begin{remark}
%The proof of Proposition \ref{prop:mult-form} produces the following more general statement. If $\omega$ is a $m$-form on $D$ which is multiplicative with respect to both groupoid structures $D\rightrightarrows H$ and $D\rightrightarrows V$, then $\nu^*\omega$ is a multiplicative $m$-form on $\barW_2 ND$.
%\end{remark}

\subsection{Nondegeneracy}\label{subsec:nondegeneracy}
Let $D$ be a presymplectic double groupoid with presymplectic form $\omega$. Even in the case where $\omega$ is nondegenerate, it is clear that the closed $2$-form $\Omega = \nu^* \omega$ on $\barW_2 ND$ is degenerate, except in the extreme case where the side groupoids $H$ and $V$ are trivial. Our aim in this section is to write down conditions, expressed only in terms of $\Omega$ and the simplicial structure of $\barW ND$, that are equivalent to the nondegeneracy of $\omega$. The idea is that such conditions are the closest thing to nondegeneracy that we should expect to hold in the definition of symplectic $2$-groupoid. Our approach is inspired by Ping Xu's definition of quasi-symplectic groupoids \cite{xu:quasi-symp}.

First, we observe that, in the case where $\omega$ is nondegenerate, then the kernel of $\Omega$ at a point $(\theta, \alpha, \eta)$ is equal to the kernel of $\nu_*$, consisting of vectors $(X,0,Y)$, where $X \in T_\theta V$ and $Y \in T_\eta H$ such that $X \in \ker s_*$ and $Y \in \ker t_*$.

Next, recall that the face maps $f^2_i: \barW_2 ND \to \barW_1 ND$ are given by \eqref{eqn:face20}--\eqref{eqn:face22}. The kernel of the push-forward map $(f^2_0)_*$ at a point $(\theta, \alpha, \eta)$ consists of vectors $(X, \gamma, 0)$, where $X \in T_\theta V$ and $\gamma \in T_\alpha D$ such that $s_* X = t^2_* \gamma$ and $\gamma \in \ker(s_V)_*$.

The degeneracy maps $\barW_1 ND \to \barW_2 ND$ are given by
\begin{align}
\delta^1_0(\theta,\eta)&=(1,1(\theta),\eta),\label{eqn:degen01} \\
\delta^1_1(\theta,\eta)&=(\theta,1(\eta),1) \label{eqn:degen11}
\end{align}
for $(\theta,\eta) \in \barW_1 ND = V \bitimes{s}{t} H$. We define a subbundle $(\delta_1^1)^* \ker \Omega \subseteq T(\barW_1 ND)$, consisting of vectors $(X,Y)$ such that $(\delta_1^1)_* (X,Y) \in \ker \Omega$. One can immediately see from \eqref{eqn:degen11} and the above description of $\ker \Omega$ that $(\delta_1^1)_* (X,Y)$ is in $\ker \Omega$ if and only if $1_* Y = 0$. Since the unit map $1$ is an embedding, we conclude that $(\delta_1^1)^* \ker \Omega = \{ (X,0) \}.$

\begin{prop}\label{prop:nondegen1}
     If $\omega$ is nondegenerate, then $(f_2^2)_*$ maps the bundle $\left( \ker(f_0^2)_* \intersect \ker \Omega \right)$ fiberwise isomorphically onto $(\delta_1^1)^* \ker \Omega$.
\end{prop}
\begin{proof}
     From the above discussion, we have that
\[  \ker(f_0^2)_* \intersect \ker \Omega = \{ (X,0,0) \} \]
and
\[ (\delta_1^1)^* \ker \Omega = \{ (X,0) \}, \]
where in both cases $X \in \ker s_* \subseteq TV$. From \eqref{eqn:face22}, we have that
\[ (f_2^2)_* (X,0,0) = (X,0), \]
so the result is clear.
\end{proof}

Similarly, we consider the subbundle $(\delta_0^1)^* \ker \Omega \subseteq T(\barW_1 ND)$, consisting of vectors $(X,Y)$ such that $(\delta_0^1)_* (X,Y) \in \ker \Omega$. One can see that $(\delta_0^1)^* \ker \Omega = \{ (0,Y) \}.$
\begin{prop}\label{prop:nondegen2}
     If $\omega$ is nondegenerate, then $(f_0^2)_*$ maps the bundle $\left( \ker(f_2^2)_* \intersect \ker \Omega \right)$ fiberwise isomorphically onto $(\delta_0^1)^* \ker \Omega$.
\end{prop}

The statements of Propositions \ref{prop:nondegen1} and \ref{prop:nondegen2} are higher-dimensional analogues of the nondegeneracy properties in Xu's definition of quasi-symplectic groupoids. These propositions give necessary conditions for the nondegeneracy of $\omega$. However, the following example demonstrates that these conditions are not sufficient.

\begin{example}\label{example:vv}
 Let $V$ be a vector space. Consider the double Lie groupoid
\[
\doublegroupoid{V}{V\oplus V^*}{pt}{V^*}.
\]
The bar construction yields the Lie $2$-groupoid
\begin{equation}\label{eqn:barvv}
    \threenerve{pt}{V\oplus V^*}{V\oplus V\oplus V^*\oplus V^*}.
\end{equation}
Actually, \eqref{eqn:barvv} is a Lie $1$-groupoid, equal to the nerve of the abelian group $V \oplus V^*$.

The symplectic structure on $V \oplus V^*$ induces a presymplectic $2$-form $\Omega$ on $V\oplus V\oplus V^* \oplus V^*$ which is supported on the middle copy of $V\oplus V^*$. On the other hand, we can check easily that \eqref{eqn:barvv} equipped with the zero $2$-form also satisfies the properties in Propositions \ref{prop:nondegen1} and \ref{prop:nondegen2}.
\end{example}

Example \ref{example:vv} shows that extra properties are needed in order to guarantee the nondegeneracy of the form $\omega$ on $D$. Our suggestion is as follows.

For $m\in M$, consider $\delta_0^0(m) = (1(m),1(m)) \in \barW_1 ND = V \bitimes{s}{t} H$. If $\omega$ is nondegenerate, then the tangent space $T_{\delta_0^0(m)}\barW_1 ND = (T_{1(m)} V) \bitimes{s_*}{t_*} (T_{1(m)} H)$ has the following three natural subspaces:
\begin{align}
W^0_0=(\delta^0_0)_*T_m M &=\{ (1_* w, 1_* w)  \suchthat w \in T_m M \}, \label{eqn:w00} \\
W^1_0=(\delta^1_0)^* \ker\Omega|_{\delta_0^0(m)} &=\{(0,Y) \suchthat Y \in \ker t_*|_{1(m)} \}, \label{eqn:w01} \\
W^1_1=(\delta^1_1)^* \ker\Omega|_{\delta_0^0(m)} &=\{(X,0) \suchthat X \in \ker s_*|_{1(m)}  \label{eqn:w11} \}.
\end{align}
These subspaces are clearly complementary:
\begin{equation}\label{eq:subspace}
T_{\delta_0^0(m)}\barW_1 ND = W^0_0\oplus W^1_0\oplus W^1_1.
\end{equation}

As we will see in Theorem \ref{thm:symplectic-2-gpd}, the statements of Propositions \ref{prop:nondegen1} and \ref{prop:nondegen2}, together with \eqref{eq:subspace}, are sufficient conditions to ensure the nondegeneracy of the $2$-form $\omega$ on $D$. Furthermore, we can recover the nondegenerate pairing between the Lie algebroids of $V$ and $H$ as follows.

Let $A^V$ and $A^H$ denote the Lie algebroids of $V$ and $H$, respectively. It is clear from \eqref{eqn:w01} and \eqref{eqn:w11} that $W_0^1$ and $W_1^1$ can be respectively identified with $A^H_m$ and $A^V_m$ at any $m \in M$. The pairing is then given by
\begin{equation}
     \langle a, b \rangle = \Omega \left( (\delta_0^1)_* a, (\delta_1^1)_* b \right)
\end{equation}
for $a \in A^V_m = W^1_1$ and $b \in A^H_m = W_0^1$. The nondegeneracy of this pairing is a consequence of \cite[Theorem 2.9]{mac:symplectic}.

\subsection{Symplectic 2-groupoid}\label{sec:sympdef}
Based on the properties discovered in \S\ref{subec:multiplicativity}--\ref{subsec:nondegeneracy}, we will now propose a definition of symplectic $2$-groupoid.

Let $X_\bullet$ be a simplicial manifold with a $2$-form $\Omega$ on
$X_2$. As in \S\ref{subsec:nondegeneracy}, we define for $i=0,1$ a (possibly singular) subbundle $(\delta_i^1)^* \ker \Omega \subseteq T X_1$, consisting of vectors $w$ for which $(\delta_i^1)_* w$ is in $\ker \Omega$.

\begin{lemma}\label{lem:map-f02}\label{lem:non-deg} If $\Omega$ is multiplicative, then
\begin{enumerate}
     \item $(f_2^2)_*$ maps $\left( \ker(f_0^2)_* \intersect \ker \Omega \right)$ into $(\delta_1^1)^* \ker \Omega$, and
     \item $(f_0^2)_*$ maps $\left( \ker(f_2^2)_* \intersect \ker \Omega \right)$ into $(\delta_0^1)^* \ker \Omega$.
\end{enumerate}
\end{lemma}
\begin{proof}
Let $v$ be a vector in $T X_2$ and $\tilde{w}$ be a vector in $TX_3$. Then, using \eqref{eqn:facedeg} and the multiplicative property of $\Omega$, we have
\begin{align*}
   \Omega( (\delta^1_1)_* (f^2_2)_*v, ( f^3_3 )_*\tilde{w}) &= \Omega \left( (f_3^3)_* (\delta_1^2)_* v , ( f^3_3 )_*\tilde{w}\right) \\
&= \left((f_3^3)^*\Omega\right) \left((\delta_1^2)_* v, \tilde{w}\right)\\
&= \left( (f_0^3)^*\Omega - (f_1^3)^*\Omega + (f_2^3)^*\Omega \right) \left((\delta_1^2)_* v, \tilde{w}\right) \\
&= \Omega \left( (f_0^3)_* (\delta_1^2)_* v, (f_0^3)_* \tilde{w} \right) - \Omega\left( (f_1^3)_* (\delta_1^2)_* v , (f_1^3)_*\tilde{w}\right)\\
&+ \Omega\left( (f_2^3)_* (\delta_1^2)_* v, (f_2^3)_*\tilde{w} \right) \\
&= \Omega \left( (\delta_0^1)_* (f_0^2)_* v , (f_0^3)_*\tilde{w}\right)+ \Omega\left( v, -(f_1^3)_*\tilde{w}+(f_2^3)_*\tilde{w} \right).
\end{align*}
Therefore, if $v\in \ker \Omega$, then we have
\[
 \Omega( (\delta^1_1)_* (f^2_2)_*v, ( f^3_3 )_*\tilde{w}) = \Omega \left( (\delta_0^1)_* (f_0^2)_* v , (f_0^3)_*\tilde{w}\right). 
\]
Now by the property that the map $(f_i^3)_*: TX_3\to TX_2$ $(i=0,...,3)$ is surjective,  if $v\in \ker(f_0^2)_*$, then $(f_2^2)_* v \in (\delta_1^1)^* \ker \Omega$, and if on the other hand $v \in \ker(f_2^2)_*$, then $(f_0^2)_* v \in (\delta_0^1)^* \ker \Omega$.

\iffalse

\begin{align*}
     \Omega \left((\delta_1^1)_* (f_2^2)_* v \right) &= \Omega \left( (f_3^3)_* (\delta_1^2)_* v \right) \\
&= \left((f_3^3)^*\Omega\right) \left((\delta_1^2)_* v\right)\\
&= \left( (f_0^3)^*\Omega - (f_1^3)^*\Omega + (f_2^3)^*\Omega \right) \left((\delta_1^2)_* v\right) \\
&= \Omega \left( (f_0^3)_* (\delta_1^2)_* v - (f_1^3)_* (\delta_1^2)_* v + (f_2^3)_* (\delta_1^2)_* v \right) \\
&= \Omega \left( (\delta_0^1)_* (f_0^2)_* v \right).
\end{align*}
Therefore, if $v \in \ker(f_0^2)_*$, then $(f_2^2)_* v \in (\delta_1^1)^* \ker \Omega$, and if on the other hand $v \in \ker(f_2^2)_*$, then $(f_0^2)_* v \in (\delta_0^1)^* \ker \Omega$.
\fi
\end{proof}

\begin{definition}\label{dfn:symp-2-gpd}
A \emph{(local) presymplectic $2$-groupoid} is a (local) Lie $2$-groupoid $X_\bullet$ equipped
with a closed multiplicative $2$-form $\Omega$ on $X_2$. A \emph{(local) symplectic $2$-groupoid} is a (local) presymplectic $2$-groupoid such that
\begin{enumerate}
\item $(f_2^2)_*$ maps $\left( \ker(f_0^2)_* \intersect \ker \Omega \right)$ fiberwise isomorphically onto $(\delta_1^1)^* \ker \Omega$,
\item $(f_0^2)_*$ maps $\left( \ker(f_2^2)_* \intersect \ker \Omega \right)$ fiberwise isomorphically onto $(\delta_0^1)^* \ker \Omega$, and
\item for all $m \in X_0$, the subspaces $W_0^0 \defequal (\delta_0^0)_* T_m X_0$ and $W_i^1 \defequal (\delta^1_i)^* \ker\Omega |_{\delta_0^0(m)}$ are such that
\[T_{\delta_0^0(m)} X_1 = W_0^0 \oplus W_0^1 \oplus W_1^1.\]
\end{enumerate}
\end{definition}

The following theorem provides some support for Definition \ref{dfn:symp-2-gpd}:
\begin{thm}\label{thm:symplectic-2-gpd}Let $D$ be a presymplectic double groupoid (see Definition \ref{dfn:presymplectic}) with presymplectic form $\omega$ on $D$. Then the local Lie $2$-groupoid $\barW ND$, equipped with the $2$-form $\Omega \defequal \nu^* \omega$ on $\barW_2 ND$ is a local symplectic $2$-groupoid if and only if $\omega$ is nondegenerate.
\end{thm}
\begin{proof}
The ``if'' part of the statement is shown in \S\ref{subsec:nondegeneracy}. For the ``only if'' part, we suppose that $\barW ND$ is a local symplectic $2$-groupoid, and we will prove that $\omega$ is nondegenerate.

For any $m \in M$, we have by definition that $W_1^1 \defequal (\delta^1_1)^* \ker\Omega |_{\delta_0^0(m)}$ consists of vectors $(X,Y)$, where $X \in T_{1(m)} V$ and $Y \in T_{1(m)} H$ are such that $s_* X = t_* Y$ and $1_* Y \in \ker \omega$. In particular, $W_1^1$ includes all vectors of the form $(X,0)$, where $s_* X = 0$.

Similarly, $W_0^1 \defequal (\delta^1_0)^* \ker\Omega |_{\delta_0^0(m)}$ consists of vectors $(X,Y)$ such that $s_* X = t_* Y$ and $1_* X \in \ker \omega$. In particular, $W_0^1$ includes all vectors of the form $(0,Y)$, where $t_* Y = 0$.

The description of $W_0^0$ is the same as in \eqref{eqn:w00}. Condition (3) in Definition \ref{dfn:symp-2-gpd} implies that $W_1^1$ only consists of vectors of the form $(X,0)$, and that $W_0^1$ only consists of vectors of the form $(0,Y)$. It follows that $\ker \omega \intersect 1_*(T_{1(m)}H) = 0$ and $\ker \omega \intersect 1_*(T_{1(m)}V) = 0$.

By \cite[Lemma 3.3]{bcwz}, we then have that
\begin{equation}
   0 = \frac{1}{2}(\dim(\ker \omega_{1^2(m)})) + 2 \dim H - \dim D
\end{equation}
and
\begin{equation}\label{eqn:symppf1}
 0 = \frac{1}{2}(\dim(\ker \omega_{1^2(m)})) + 2 \dim V - \dim D.
\end{equation}
Here, we use $1^2$ to denote the iterated unit map from $M$ to $D$.

Now consider a point of the form $(1(m), 1(\theta), \eta) \in \barW_2 ND$ where $\theta \in V$ is such that $t(\theta) = m$. By definition, $(\ker (f_0^2)_* \intersect \ker \Omega)|_{(1(m), 1(\theta), \eta)}$ consists of vectors $(X, \gamma, 0)$, where $X \in T_{1(m)} V$ and $\gamma \in T_{1(\theta)} D$ are such that $s_* X = t^2_* \gamma$ and $\gamma \in \ker(s_V)_* \intersect \ker \omega$. On the other hand, condition (1) in Definition \ref{dfn:symp-2-gpd} implies that $(\ker (f_0^2)_* \intersect \ker \Omega)|_{(1(m), 1(\theta), \eta)}$ consists only of vectors of the form $(X, 0, 0)$. It follows that $(\ker(s_V)_* \intersect \ker \omega)|_{1(\theta)} = 0$. Again using \cite[Lemma 3.3]{bcwz}, we have that
\begin{equation}\label{eqn:symppf2}
 0 = \frac{1}{2}(\dim(\ker \omega_{1(\theta)})) - 2 \dim V + \dim D.
\end{equation}
Similarly, we can see that for any $\eta \in H$,
\begin{equation}
0 = \frac{1}{2}(\dim(\ker \omega_{1(\eta)})) - 2 \dim H + \dim D.
\end{equation}
Comparing \eqref{eqn:symppf1} and \eqref{eqn:symppf2} with $\theta = 1(m)$, we have that $\dim D = 2\dim V$, so we conclude from \eqref{eqn:symppf2} that $\omega$ is nondegenerate at all unit elements $1(\theta)$. It then follows from the fact that $\omega$ is multiplicative (see, e.g.\ \cite[Lemma 4.2]{bcwz}) that $\omega$ is nondegenerate everywhere.
\end{proof}

As mentioned in the introduction, the construction of symplectic
$2$-groupoids from symplectic double groupoids allows us to
integrate the standard Courant algebroid $TM \oplus T^*M$ to the
symplectic $2$-groupoid \eqref{eqn:standard2}. However, the Courant
algebroid structures on $TM \oplus T^*M$ that are twisted by closed
$3$-forms cannot be integrated via symplectic double groupoids, and
it is not clear how (or even if) such Courant algebroids can be seen
to integrate\footnote{Sheng and Zhu communicated \cite{sh-zh} to us
that they have some ideas to integrate $TM \oplus T^*M$ twisted by
nontrivial closed 3-forms to Lie 2-groupoids using representations
up to homotopy as in \cite{sh-zhu:semi} and
\cite{sh-zh:integration}.} to a symplectic $2$-groupoid in the sense
of Definition \ref{dfn:symp-2-gpd}.

The definitions suggested to us by Li-Bland and \v{Severa}
\cite{lbs} include a $3$-form $\Omega_1$ on $X_1$ in addition to a
$2$-form $\Omega_2$ on $X_2$, such that $\Omega_1 + \Omega_2$ is a
closed element of the total complex of simplicial forms. It seems
quite likely that a generalization of Definition
\ref{dfn:symp-2-gpd} along these lines will be necessary to
incorporate all Courant algebroids.

\bibliography{dbl-bib}
\bibliographystyle{amsalpha}
\end{document}